\documentclass[10pt, article]{amsart}

\usepackage{tikz}
\usetikzlibrary{calc}
\usepackage{ae} 
\usepackage[T1]{fontenc}
\usepackage[cp1250]{inputenc}
\usepackage{amsmath}
\usepackage{amssymb, amsfonts,amscd,verbatim}
\usepackage{mathtools}

\usepackage[normalem]{ulem}
\usepackage{hyperref}
\usepackage{indentfirst}
\usepackage{latexsym}
\input xy
\xyoption{all}

\usepackage{amsmath}    

\theoremstyle{plain}
\newtheorem{Pocz}{Poczatek}[section]
\newtheorem{Proposition}[Pocz]{Proposition}

\newtheorem{Theorem}[Pocz]{Theorem}

\newtheorem{Question}[Pocz]{Question}

\theoremstyle{definition}
\newtheorem{Definition}[Pocz]{Definition}

\theoremstyle{remark}
\newtheorem{Remark}[Pocz]{Remark}

\def\RR{{\mathbb R}}

\def\NN{{\mathbb N}}

\numberwithin{equation}{section}

\title[Discreteness and Large Scale Surjections
]%
  {Discreteness and Large Scale Surjections}

\author{Kyle ~Austin}
\address{University of Tennessee, Knoxville, USA}
\email{austin@math.utk.edu}

\date{ \today
}


\begin{document}
\maketitle

\begin{abstract}
We study the concept of coarse disjointness and large scale $n$-to-$1$ functions. As a byproduct, we obtain an Ostrand-type characterization of asymptotic dimension for coarse structures. It is shown that properties like finite asymptotic dimension, coarse finitism, large scale weak paracompactness, ect. are all invariants of coarsely $n$-to-$1$ functions. Metrizability of large scale structures is also investigated.
\end{abstract}

\section{Introduction}

Disjointness and discreteness have long been useful topological concepts. A few examples are Ostrand's characterization of dimension, Bing's Metrization Theorem, and using discreteness to find partitions of unity \cite{AD}. Dimension theorists and geometric group theorists have been effectively using coarse disjointness in the large scale category for quite some time. J. Dydak in \cite{DyCAD} explores the concept of disjointness in the large scale category. He shows that certain classes of large scale spaces can be characterized via coarse disjointness. Roughly speaking, the first half of this paper is devoted to the investigation of disjointness in large scale structures. The highlight of disjointness results in this paper is an Ostrand dimension characterization for large scale structures.

T. Miyata and \v Z. Virk introduced coarse analogs of the Hurewicz dimension raising theorems in \cite{MV}. There they showed that finite asymptotic dimension was preserved by functions which have a property that is coarsely analogous to $n$-to-$1$ functions which they called $(B)_n$. In particular, they proved that for  a metric space $X$, the property that $asdim(X) \leq n$ is equivalent to the existence of a $(B)_n$ function $f:Y\to X$ from a space $Y$ with $asdim(Y) = 0$. Their results show that there are analogues to the classical $n$-to-$1$ maps from the Cantor set onto an $n$ dimensional compact space. We aim to study these functions that satisfy the $(B)_n$ property which we will call coarsely $n$-to-$1$ functions.

The large scale analogue of a surjection in coarse geometry is a function that becomes indistinguishable from a  surjection when viewed from ever increasing distances. More precisely, a function $f:X \to Y$ of metric spaces $X$ and $Y$ is coarsely surjective if the image of $X$ is an $R$-net in $Y$ for some $R>0$. In order to define what a coarsely $n$-to-$1$ map of metric spaces is, we need to generalize the notion of a point. The points of a metric space $X$ are any collection of subsets of $X$ which become indistinguishable from points when viewed from ever increasing distances. So points of $X$ look like a uniformly bounded family of subsets of $X$. A function $f:X \to Y$ of metric spaces is \textbf{coarsely} $n$-to-$1$ if for every $R > 0$ there exists $S>0$ such that the point inverse of any set of diameter at most $R$ can be covered by at most $n$ sets of diameter at most $S$. 

The main resluts of this paper are as follows:

\begin{Theorem}[Ostrand Characterization of Dimension] \nonumber
Let $(X,\mathcal{LSS}_X)$ be a  large scale structure and $\mathcal{U}$ a uniformly bounded cover of $X$. The following are equivalent:

1) There exists a uniformly bounded cover $\mathcal{V} = \bigcup_{i = 1}^{n+1} \mathcal{V}_i$ where $st(\mathcal{V}_i,\mathcal{U})$ is a disjoint collection for each $i = 1,2,\hdots , n$.

2) There exists a uniformly bounded cover $\mathcal{V}$ which coarsens $\mathcal{U}$ with multiplicity at most $n+1$.

\end{Theorem}

\begin{Theorem}
Let $X$ and $Y$ be coarse structures and $f:X\to Y$ coarse and coarsely $n$-to-$1$ and satisfies the following additional property:

($*$) For every uniformly bounded cover $\mathcal{U}$ there is a uniformly bounded cover $V$ of $X$ such that for each element $U \in \mathcal{U}$ there exists $U_1,U_2,\hdots,U_n \in \mathcal{U}$ such that $st^n(f(U_i),f(\mathcal{V})) \supset V$ for each $1\leq i \leq n$. 

If $X$ then $Y$ is metrizable. 
\end{Theorem}

\begin{Theorem}
Let $f:X\to Y$  be coarse and coarsely $n$-to-$1$ between general large scale structures $X$ and $Y$. $X$ has finite asymptotic dimension if and only if $Y$ has finite asymptotic dimension. 
\end{Theorem}

\begin{Theorem}
Let $f:X\to Y$  be coarse and coarsely $n$-to-$1$. $X$ is large scale finitistic if and only if $Y$ is large scale finitistic. 
\end{Theorem}

\begin{Theorem}
Let $X$ and $Y$ be spaces and $f:X \to Y$ coarse and coarsely finite-to-$1$. If $X$ is large scale weakly paracompact then so is $Y$.
\end{Theorem}

\begin{Theorem}
Let $X$ and $Y$ be spaces and $f:X \to Y$ coarse and large scale $n$-to-1. If $X$ is of bounded geometry then $Y$ is of bounded geometry. Conversely, if $Y$ is of bounded geometry then there exists a bounded geometry subspace $Z\subset X$ for which the inclusion induces a coarse equivalence from $Z$ to $X$.
\end{Theorem}

All of these concepts are introduced for the metric setting in the survey \cite{CDV1}(see also \cite{CDV2}). In the process of proving the above assertions, we show that large scale $n$-to-$1$ functions are optimal for pushing forward certain discrete collections of sets. It is therefore another goal of this paper to generalize certain discreteness properties to large scale structures. 

The author would like to thank Jerzy Dydak for all of his ideas and suggestions for improving this note. The author would like to also thank \v Ziga Virk and Michael Holloway for many helpful discussions on the topic.  

\section{Surjections of Large Scale Structures}\label{2}

In \cite{DH} the authors define \textbf{Large Scale Structures} for any set $X$. Their definition is equivalent to that of a coarse structure by John Roe in \cite{R}. Their motivation was that the definition of coarse structures given by Roe \cite{R} should be equivalent to specifying which collections of subsets of $X$ are uniformly bounded. Before getting to the definition, we introduce some terminology in order to better motivate the definition of large scale structure.

Given some set $X$ and $\mathcal{U} \subset 2^X$(the power set of $X$) the \textbf{star} of some subset $U$ of $X$ with respect to $\mathcal{U}$ is defined by by $st(U,\mathcal{U}) = \cup\{V \in \mathcal{U}: V \cap U \neq \emptyset\}$. Given two collections $\mathcal{V}_1 ,\mathcal{V}_2 \subset 2^X$ the star of $\mathcal{V}_1$ with respect to $\mathcal{V}_2$ is  denoted by $st(\mathcal{V}_1,\mathcal{V}_2)$  and is the new collection $\{st(V,\mathcal{V}_2): V \in \mathcal{V}_1\}$. We define $st(\mathcal{U}) = st(\mathcal{U},\mathcal{U})$ and inductively define higher stars by $st^n(\mathcal{U}) = st(st^{n-1}(\mathcal{U}),\mathcal{U})$. Given two covers $\mathcal{U}$ and $\mathcal{V}$ of some set $X$, we say that $\mathcal{U}$ \textbf{coarsens} $\mathcal{V}$(equivalently $\mathcal{V}$ \textbf{refines} $\mathcal{U}$), denoted by $\mathcal{U} \ge\mathcal{V}$, if each element of $\mathcal{V}$ is contained in some element of $\mathcal{U}$. Recall that the \textbf{Lebesgue number} of a cover $\mathcal{U}$ of a metric space $X$ is $\sup\{R \in \RR_{\ge 0}: \mathcal{U} \ge \{B(x,R):x\in X\}\}$ (this is actually two times the usual Lebesgue number). Also recall that the \textbf{multiplicity} of a cover $\mathcal{U}$ of some set $X$ is the maximum number of elements of $\mathcal{U}$ which contain a point in common and $\infty$ if no such maximum exists.

Notice that if $\mathcal{U}$ is a cover of some set, then $st(\mathcal{U},\mathcal{U})$ is a cover whose elements are unions of the elements of $\mathcal{U}$. We may think of $\mathcal{U}$ like it is a collection of points and $st(\mathcal{U},\mathcal{U})$ as a collection of neighborhoods of the points of $\mathcal{U}$. The philosophy behind coarse geometry is that every cover is a cover by points (when you zoom away far enough) and that you have to work your way outward via staring to find neighborhoods of those points. Here is the precise definition of a large scale structure.

\begin{Definition}[\cite{DH}]
A \textbf{Large Scale Structure} on a set $X$ is a nonempty set of covers $\mathcal{B}$ of subsets of $X$ satisfying 

1) $\mathcal{LSS}$ is closed under refinements.

2) If $\mathcal{B}_1$, $\mathcal{B}_2 \in \mathcal{B}$ then $st(\mathcal{B}_1,\mathcal{B}_2) \in \mathcal{B}$. 

\end{Definition}

\begin{Remark}
If a set $X$ is endowed with a large scale structure $\mathcal{LSS}$, then we will often just say that $X$ is a \textbf{large scale space} without making mention of $\mathcal{LSS}$.
\end{Remark}

Let $(X,\mathcal{LSS}_X)$ be a large scale structure. Given a uniformly bounded family $\mathcal{B} \in \mathcal{LSS}_X$ we define the trivial extension of $\mathcal{B}$ to be $\mathcal{B} \cup \{\{x\}:x\in X\}$. By $1$ above and the fact that there is at least one cover in $\mathcal{LSS}$, we have that the collection of singletons is in any large scale structure. 

An important example of large scale structures that will be used in the last section is the metric large scale structure. If $X$ is metric then the uniformly bounded families are collections of subsets of $X$ with bounded mesh; i.e the collections are precisely the refinements of the covers of $X$ by $R$-balls where $R\ge 0$. We begin to motivate the definition of coarsely $n$-to-$1$ maps using this large scale structure as a base model. Recall that for a metric $X$ and $R>0$ a collection of subsets $\mathcal{B}$ of $X$ is said to be \textbf{$R-$discrete} if $dist(U,V)\ge R$ for all $U,V \in \mathcal{B}$. The proof of the following proposition can be found in \cite{DV}.

\begin{Proposition}\label{Discrete} 
A function $f:X\to Y$ of metric spaces is coarsely $n$-to-$1$ if and only if for every $R,S>0$ there is a uniformly bounded cover $\mathcal{V}$ of $X$ such that the preimage of an $S-$ball in $Y$ can be covered by at most $n$ elements of $\mathcal{V}$ that are $R-$disjoint.
\end{Proposition}
\begin{proof}
$\Leftarrow$ is straightforward.\\
 $(\Rightarrow)$ Let $\mathcal{U}$ be a collection of subsets of a metric space. We define an $R$-lattice in $\mathcal{U}$ is a subset $\mathcal{V} \subset \mathcal{V}$ such that for each $V\in \mathcal{V}$, there exists $V' \in \mathcal{V}$ such that $d(V,V') \leq R$. It is easy to see that the $R-$lattices partition $\mathcal{U}$.

 Let $\mathcal{U}$ be a cover of $Y$ by $S-$points and let $\mathcal{V}$ a collection of $R$-points be such that the preimage of an $S$-point $A \i \mathcal{U}$ can be covered by at most $n$ $R-$points $A_1,A_2,\hdots ,A_n \in \mathcal{V}$ in $X$. Let $\{\mathcal{B}_j: 1 \leq j \leq n\}$ be a partition $\{A_1,A_2,\hdots ,A_n\}$ into $R$-lattices. Replace $\{A_1,A_2,\hdots ,A_n\}$ with $\{\cup\mathcal{B}_i: 1\leq j \leq n\}$ and notice that it is $R-$disjoint and is coarsened by the cover $st^n(\mathcal{V})$. It follows that there exists a uniformly bounded cover (uniformly bounded by $st^n(\mathcal{V})$) that satisfies the desired properties.
\end{proof}

The following proposition sheds some light on how to define $R$-disjointness in the general large scale setting.  Recall that a metric space has the \textbf{midpoint property} if for every $x,y \in X$ there exists a point $z \in X$ such that $d(x,z) = d(y,z) = \frac{d(x,y)}{2}$. 

\begin{Proposition}
Let $X$ be a metric space with midpoint property. A collection of closed subsets $\mathcal{V}$ is $R-$disjoint if and only if the collection $st(\mathcal{V},\{B(x,\frac{R}{2}):x\in X\})$ is a disjoint family. If $X$ does not have the midpoint property then 
\end{Proposition}
\begin{proof}
$(\Rightarrow)$ If $\mathcal{V}$ is $R$-disjoint then the collection $st(\mathcal{V},\{B(x,\frac{R}{2}):x\in X\})$ is disjoint by the triangle inequality.

$(\Leftarrow)$ Assume $st(\mathcal{V},\{B(x,\frac{R}{2}):x\in X\})$ is a disjoint family and suppose there is $A,B \in \mathcal{V}$ with $d(A,B) < R$. It follows that there is $x \in A$ and $y \in B$ such that $d(x,y) < R$. Let $z$ be a midpoint between $x$ and $y$. Notice then that $z \in st(A, \{B(x,\frac{R}{2}):x\in X\})$ and $z \in st(B, \{B(x,\frac{R}{2}):x\in X\})$ contrary to our assumption.
\end{proof}

The following proposition shows that one does not need to consider only geodesic metric spaces.

\begin{Proposition}
Let $X$ be metric. The following are equivalent for collections of subsets $\mathcal{B}_i$:\\
1) There exists a sequence of $\{R_i\}$ diverging to $\infty$ such that $\mathcal{B}_i$ is $R_i$-disjoint for each $i \ge 1$. \\
2) There exists a sequence of covers $\mathcal{U}_i$ such that $Leb(\mathcal{U}_i)$ diverges to $\infty$ and $st(\mathcal{B}_i,\mathcal{U}_i)$ is a disjoint collection for each $i\ge 1$.
\end{Proposition}
\begin{proof}
(1) $\Rightarrow$ (2) It is easy to see that $(1)$ implies that $st(\mathcal{B}_i, \{B(x,R_i/2):x\in X\})$ is disjoint for $i \ge 1$.\\
(2) $\Rightarrow$ (1) By (2), we know that $st(\mathcal{B}_i,\{B(x,Leb(\mathcal{U}_i)):x\in X\})$ is a disjoint family and it therefore follows that $\mathcal{B}_i$ is at least $Leb(\mathcal{U}_i)$-disjoint for each $i\ge 1$.
\end{proof}

In view of the above proposition, it makes sense to define coarsely $n$-to-$1$ maps in two ways for general large scale structures. By the lemma, the notion of $R$-discreteness can be generalized to large scale structures as follows: Let $\mathcal{W}$ be a uniformly bounded collection of a large scale structure $X$. A collection $\mathcal{V}$ is said to be \textbf{$\mathcal{W}-$discrete} if $st(\mathcal{V},\mathcal{W})$ is a disjoint family.

\begin{Definition}
Let $X$ and  $Y$ be large scale structures on sets $X$ and $Y$. A function $f:X\to Y$ is \textbf{coarse}(or \textbf{bornologous}) if for every $\mathcal{B}_X \in \mathcal{LSS}_X$ there exists some $\mathcal{B}_Y \in \mathcal{LSS}_Y$ such that $f(\mathcal{B}_x)$ refines $\mathcal{B}_Y$.

A function $f:X \to Y$ is \textbf{coarsely $n$-to-$1$}(or \textbf{large scale $n$-to-$1$}) if for every uniformly bounded cover $\mathcal{U}_Y$ of $Y$ there exists a uniformly bounded cover $\mathcal{U}_X$ of $X$ such that for every $B\in \mathcal{U}_Y$ there exists $B_1,B_2, \hdots, B_n \in \mathcal{U}_X$ with $f^{-1}(B) \subset \bigcup_{i=1}^n B_i$. 

A function $f:X \to Y$ is \textbf{discretely $n$-to-$1$} if for every uniformly bounded cover $\mathcal{U}_Y$ of $Y$ and uniformly bounded cover $\mathcal{W}_X$ of $X$ there exists a uniformly bounded cover $\mathcal{U}_X$ of $X$ such that for every $B\in \mathcal{U}_Y$ there exists a $\mathcal{W}_X$-discrete collection $\{B_1,B_2, \hdots, B_n\} \subset \mathcal{U_X}$  with $f^{-1}(B) \subset \bigcup_{i=1}^n B_i$.
\end{Definition}

\begin{Remark} This definition of a coarse map does not assume that the preimage of a bounded set is bounded. 
\end{Remark}

\begin{Remark}
We will later show that discretely $n$-to-$1$ maps are the same as coarsely $n$-to-$1$ maps.
\end{Remark}

The following proposition will be useful later on and its proof is imediate.

\begin{Proposition}
Let $f:X\to Y$ be a function between large scale structures. $f$ is coarsely $n$-to-$1$ if and only if for every uniformly bounded cover $\mathcal{U}_Y$ of $Y$ there exists a uniformly bounded cover $\mathcal{U}_X$ of $X$ such that for every $B\in \mathcal{U}_Y$ there exists $B_1,B_2, \hdots, B_n \in \mathcal{U}_X$ with $f^{-1}(B)= \bigcup_{i=1}^n B_i$. 
\end{Proposition}

\section{Metrizability}

In this section, we show that metrizability is pushed forward by a certain class of coarsely $n$-to-$1$ maps. We recall some results about metrizability for large scale structures.

\begin{Proposition} \label{LargeScaleBasis} \cite{DH}
If $\mathcal{LSS}_X'$ is a set of families of subsets of $X$ such that $\mathcal{B}_{\alpha},\mathcal{B}_{\beta} \in \mathcal{LSS}_X'$ implies the existence of $\mathcal{B}_{\gamma} \in \mathcal{LSS}_X'$ such that $\mathcal{B}_{\alpha} \cup \mathcal{B}_{\beta} \cup st(\mathcal{B}_{\alpha},\mathcal{B}_{\beta})$ refines $\mathcal{B}_{\gamma}$, then the family $\mathcal{LSS}_X$ of all refinements of trivial extensions of elements of $\mathcal{LSS}_X'$ forms a large scale structure on $X$.
\end{Proposition}

A standard set of families that satisfy the above criterion if the collection of covers of a metric space by $n-$balls for $n \in \NN$. The large scale structure of a metric space is precisely the the set of all refinements of these covers. It will be convenient to define large scale structures as those generated by a certain collection of uniformly bounded covers just as one defines a topology by specifying a basis. This leads to the following definition.

\begin{Definition}
Let $(X,\mathcal{LSS}_X)$ be a large scale structure. If there exists a set of families $\mathcal{LSS}$ which satisfy the criterion of proposition \ref{LargeScaleBasis} above then we will say that $\mathcal{LSS}$ is a \textbf{large scale basis} for $(X,\mathcal{LSS}_X)$. 
\end{Definition}

The following is a nice characterization of metrizability of a LS-structure in terms of large scale basis. The proof can be found in \cite{DH} and in \cite{R}.

\begin{Proposition} \label{metrizability} \cite{DH}
Let $(X,\mathcal{LSS}_X)$ be a large scale structure. The following are equivalent:

1) $(X,\mathcal{LSS}_X)$ has a countable large scale basis $\mathcal{LSS}_X'$

2) There exists an $\infty$ metric $d$ such that the uniformly bounded covers with respect to the metric coincide with $\mathcal{LSS}_X$.
\end{Proposition} 

\begin{Definition}
Let $(X,\mathcal{LSS}_X)$ be a large scale structure. Then $X$ is called \textbf{metrizable} if $X$ admits an $\infty$ metric such that $\mathcal{LSS}_X$ consists of all uniformly bounded collections of subsets of $X$
\end{Definition}

\begin{Theorem}
Let $X$ and $Y$ be coarse structures and $f:X\to Y$ coarse and coarsely $n$-to-$1$ and satisfies the following additional property:

($*$) For every uniformly bounded cover $\mathcal{U}$ there is a uniformly bounded cover $V$ of $X$ such that for each element $U \in \mathcal{U}$ there exists $U_1,U_2,\hdots,U_n \in \mathcal{U}$ such that $st^n(f(U_i),f(\mathcal{V})) \supset V$ for each $1\leq i \leq n$. 

If $X$ then $Y$ is metrizable. 
\end{Theorem}
\begin{proof}
By Theorem \ref{metrizability} we need only prove that $\mathcal{LSS}_X$ has a countable large scale basis if and only if $\mathcal{LSS}_Y$ has a countable large scale basis. 

Let $\mathcal{LSS}_X'$ be a countable large scale basis for $X$. Consider the collection $\mathcal{LSS}_Y'$ consisting of all possible finite stars of elements of $f(\mathcal{LSS}_X') := \{f(\mathcal{B}):\mathcal{B} \in \mathcal{LSS}_X\}$ where $f(\mathcal{B}) := \{f(B):B\in \mathcal{B}\}$.  $\mathcal{LSS}_Y'$ is a countable collection for $Y$ which satisfies the finite additivity condition of proposition \ref{LargeScaleBasis}. It needs to be shown that the large scale structure on $Y$ generated by $\mathcal{LSS}_Y'$ is $\mathcal{LSS}_Y$. It suffices to show that every $\mathcal{B} \in \mathcal{LSS}_Y$ refines some collection in $\mathcal{LSS}_Y'$. 

Let $\mathcal{B}_Y \in \mathcal{LSS}_Y$. We have $\mathcal{B}_X \in \mathcal{LSS}_X$ such that for each $B \in \mathcal{B}_Y$ there is $B_1,B_2, \hdots B_n \in \mathcal{B}_X$ with $B \subset \bigcup_{i=1}^nB_i$ such that $st^n(f(U_i),f(\mathcal{V})) \supset V$ for each $1\leq i \leq n$.  Let $\mathcal{B}_X' \in \mathcal{LSS}_X'$ be a coarsening of $\mathcal{B}_X$. Notice that $st^n(f(\mathcal{B}_X'))$ coarsens the cover $\mathcal{B}_Y$ and $st^n(f(\mathcal{B}_X')) \in \mathcal{LSS}_Y'$ which completes the claim.
\end{proof}

\begin{Question}
Let $X$ and $Y$ be large scale structures and $f:X \to Y$ coarse. Is $f$ is coarsely $n$-to-$1$ if and only if for every uniformly bounded cover $\mathcal{U}$ there is a uniformly bounded cover $V$ of $X$ such that for each element $U \in \mathcal{U}$ there exists $U_1,U_2,\hdots,U_n \in \mathcal{U}$ such that $st^n(f(U_i),f(\mathcal{V})) \supset V$ for each $1\leq i \leq n$?
\end{Question}

\begin{Proposition}
Let $X$ and $Y$ be large scale structures and $f:X\to Y$ a function. $f$ is discretely $n$-to-$1$ if and only if $f$ is coarsely $n$-to-$1$.
\end{Proposition}
\begin{proof}
$(\Rightarrow)$ is clear.

$(\Leftarrow)$ Let $\mathcal{U}$ be a uniformly bounded cover of $Y$ and $\mathcal{W}_X$ a uniformly bounded cover of $X$. Let $\mathcal{V}_0$ be a uniformly bounded cover of $X$ such that for every $B\in \mathcal{U}$ there exists $B_1,B_2, \hdots, B_n \in \mathcal{V}$ with $f^{-1}(B) \subset \bigcup_{i=1}^n B_i$. Inductively define $\mathcal{V}_{n}$ as a uniformly bounded cover of $X$ such that for each element $U \in st^n(\mathcal{U})$ there exists $U_1,U_2, \hdots, U_n \in \mathcal{V}_n$ with $f^{-1}(U) \subset \bigcup_{i=1}^n U_i$. 

$\{st^n(\mathcal{U})\}$ is a large scale basis for a metrizable large scale structure $(Y,d)$. The collection of all finite stars of elements of the collection $\{\mathcal{W}_X,\mathcal{V}_n:n\ge 0\}$ is a large scale basis for a metrizable large scale structure $(X,d')$. $(X,d')$ and $(Y,d)$ have been designed so that $f:(X,d') \to (Y,d)$ is coarsely $n$-to-$1$ which means, by proposition \ref{Discrete}, that there exists a uniformly bounded cover $\mathcal{W}$ of $(X,d)$ such that  for every $B\in \mathcal{U}$ there exists a $\mathcal{W}_X$-discrete collection $\{B_1,B_2, \hdots, B_n\} \subset \mathcal{U_X}$  with $f^{-1}(B) \subset \bigcup_{i=1}^n B_i$.
\end{proof}

\section{Spaces of Bounded Geometry}

The purpose of this section is show that large scale $n$-to-$1$ maps preserve the property of being of bounded geometry. To cut down on wordiness we will use the following termonology: Let $\mathcal{U}$ be a cover of some set $X$. An $\mathcal{U}$-point in  $X$ is a subset $A\subset X$ which is contained in an element of $\mathcal{U}$.

\begin{Definition}
A large scale space $X$ is said to have $\textbf{bounded geometry}$  if for every uniformly bounded cover $\mathcal{U}$ there is $m(\mathcal{U})>0$ so that each $\mathcal{U}$-point contains no more than $m(\mathcal{U})$ elements. 
\end{Definition}

Typical examples of bouded geometry spaces of finitely generated groups with the Cayley graph metric. In view of the Svar\v c Milnor Lemma, this gives a wealth of examples of metric spaces $X$ which admit proper and cocompact actions by finitely generated groups.

Recall that a map $f:X\to Y$ of metric spaces is called a \textbf{coarse embedding} if there exists nondecreasing functions $p^+_-:[0,\infty) \to [0,\infty)$ such that $p_-(d(x,y)) \leq d(f(x),f(y)) \leq p^+(d(x,y))$; i.e. $f$ preserves the coarsening of coverings. If is a \textbf{coarse equivalence} if the image of $X$ is an $R$ net in $Y$ for some $R>0$. Note in particular that an inclusion $i:X\xhookrightarrow{} Y$ is a coarse equivalence if and only if $X$ is an $R$-net in $Y$ for some $R>0$.

A reason to consider spaces of bounded geometry is that the the Rips Complex simplicial approxiations are locally finite and hence metrizable. Typically, a metric space $X$ is approximated by Rips complexes of a coarsely equivalent bounded geometry subspace. The main result of this section says that if there exists a large scale $n$-to-$1$ function $f:X\to Y$ then $X$ can be approximated by metrizable Rips complexes of coarsely equivalent subspaces if and only if the same is true for $Y$.

\begin{Theorem}
Let $X$ and $Y$ be large scale spaces and $f:X \to Y$ coarse and large scale $n$-to-1. If $X$ is of bounded geometry then $Y$ is of bounded geometry. Conversely, if $X$ and $Y$ are both metric spaces and if $Y$ is of bounded geometry then there exists a bounded geometry subspace $Z\subset X$ for which the inclusion induces a coarse equivalence from $Z$ to $X$.
\end{Theorem}

\begin{proof}
($\Rightarrow$) Let $\mathcal{V}$ be a uniformly bounded cover of $Y$ and let $\mathcal{U}$ be a uniformly bounded cover of $X$ such that point inverses of elements of $\mathcal{V}$ can be covered by at most $n$ elements of $\mathcal{U}$. There exists an upper bound $m(\mathcal{U})$ to the number of elements of a set in $\mathcal{U}$. Notice then that the cardinality of an element of $\mathcal{V}$ is bounded above by $n\cdot m(\mathcal{U})$.

($\Leftarrow$) 
Given $R>0$ there exists $m>0$ such that the inverse image of $R$-points in $Y$ can be covered by $n$ $m$-points in $X$. For each $y\in Y$ we define the subset $R_y \subset X$ to be a set containing one point from each of the $m$-points covering $f^{-1}(B(y,R))$.

Consider $Z = \bigcup_{y\in Y} R_y$. To see that $Z$ is of bounded geometry, let $n>0$ and $\mathcal{U}_n$ be the uniformly bounded cover of $Z$ by $n$-balls and consider $|U|$, the cardinality of $U$, for some $U\in \mathcal{U}_n$. Notice that $f$ takes $\mathcal{U}_n$ to a uniformly bounded cover of $Y$ which can be coarsened by the uniformly bounded cover of $Y$ by $S$-balls for some $S > R$. $Y$ is of bounded geometry(and hence is coarsely doubling, see \cite{CDV1}) which means that there is $l(S)>0$ for which each ball of radius $S$ can be covered by at most $l(S)$ $R$-points. Let $p(R)$ be the maximum number of elements contained in a $R$-point in $Y$. Thus the image of $U$ lies in a $S$-ball which has has at most $l(S)\cdot p(R)$ points. If follows that $U$ contains at most $n \cdot l(S) \cdot p(R)$ points and so $Z$ is of bounded geometry.

To see that $Z$ is coarsely equivalent to $X$, just observe that $Z$ forms a net in $X$. Indeed, every element of the cover by $2m$ balls of $X$ contains a point of $Z$.
\end{proof} 

\section{Asymptotic Dimension and Finitism}
The purpose of this section is to translate some ideas from \cite{MV} into the language of large scale structures. We will motivate our definition of asymptotic dimension with the metric case. Let $X$ be metric and $n \ge 0$ an integer. 

1) $X$ is said to have \textbf{asymptotic dimension at most $n$} provided that for each $R>0$ there exists a uniformly bounded cover of $X$ with Lebesgue number greater than $R$ and having multiplicity at most $n+1$. 

Following \cite{MV}, we opt for different definition of asymptotic dimension.

2) $X$ is said to have \textbf{asymptotic dimension at most $n$} provided that for each $R>0$ there exists a uniformly bounded cover $\mathcal{V} = \bigcup_{i=1}^{n+1}\mathcal{V}_i$ where $\mathcal{V}_i$ is an $R$-disjoint family for each $i=1,2,\hdots, n+1$.

Again, based on the equivalence of the previous two definitions in the metric case, it makes sense to define two notions of asymptotic dimension for large scale structures. 

\begin{Definition}
A large scale structure $X$ is said to have \textbf{asymptotic dimension} at most $n$, denoted by $asdim(X) \leq n$ if For every uniformly bounded cover $\mathcal{U}$ of $X$ there exists a uniformly bounded cover $\mathcal{V}$ coarsening $\mathcal{U}$ with multiplicity at most $n+1$.

A large scale structure $X$ is said to be \textbf{large scale finitistic} if for every uniformly bounded cover $\mathcal{U}$ of $X$ there exists $m \ge 1$ and a uniformly bounded cover $\mathcal{V}$ which coarsens $\mathcal{U}$ with multiplicity at most $m$.
\end{Definition}

The forward direction of the following proof is an adaptation of the proof for metric spaces found in \cite{BD}. There G. Bell and A. Dranishnikov prove that these two definitions coincide for the case of metric spaces. The idea for the converse was suggested to the author by Jerzy Dydak. We obtain that the two criterion aforementioned for metric spaces are also equivalent in the class of general large scale structures.

\begin{Theorem}[Ostrand Characterization]
Let $(X,\mathcal{LSS}_X)$ be a  large scale structure and $\mathcal{U}$ a uniformly bounded cover of $X$. The following are equivalent:

1) There exists a uniformly bounded cover $\mathcal{V} = \bigcup_{i = 1}^{n+1} \mathcal{V}_i$ where $st(\mathcal{V}_i,\mathcal{U})$ is a disjoint collection for each $i = 1,2,\hdots , n$.

2) There exists a uniformly bounded cover $\mathcal{V}$ which coarsens $\mathcal{U}$ with multiplicity at most $n+1$.

\end{Theorem}
\begin{proof}
($\Rightarrow$) Let $\mathcal{U}$ be a uniformly bounded cover of $X$. Choose a uniformly bounded cover $\mathcal{V}= \bigcup_{i = 1}^{n+1} \mathcal{V}_i$ where $st(\mathcal{V}_i,st(\mathcal{U},\mathcal{U}))$ is a disjoint collection for each $i = 1,2,\hdots , n$. Consider the cover $\mathcal{W} = st(\mathcal{V},\mathcal{U})$. This is a coarsening of $\mathcal{U}$ and we claim that the multiplicity is at most $n+1$. Notice that $st(\mathcal{V}_i,\mathcal{U})$ is disjoint for each $i = 1,2,\hdots n$ which means that each element of $X$ lies in at most one element of $st(\mathcal{V}_i,\mathcal{U}) \subset \mathcal{W}$. It follows that each element of $X$ belongs to at most $n+1$ elements of $\mathcal{W}$.

$(\Leftarrow)$ Let $\mathcal{U}$ be a uniformly bounded cover of $X$. We need to construct a metric model that captures the dimension for $X$ in terms of $\mathcal{U}$. Let $\mathcal{U}_0 = \mathcal{U}$. Let $\mathcal{U}_1$ be a uniformly bounded cover of $X$ which coarsens $\mathcal{U}_0$ and has multiplicity at most $n+1$. Define $\mathcal{U}_2 = st(\mathcal{U}_1,\mathcal{U}_0)$. Notice that this cover coarsens both $\mathcal{U}_0$ and $\mathcal{U}_1$. Let $\mathcal{U}_3$ be a uniformly bounded cover of $X$ which coarsens $\mathcal{U}_2$ and has multiplicity at most $n+1$. Continue as follows: for $k \ge 3$ define $U_{k}$ as a uniformly bounded cover of $X$ which coarsens $st(\mathcal{U}_{k-1},\mathcal{U}_{k-2})$ and has multiplicity at most $n+1$. Observe that $\mathcal{U}_{i+1}$ coarsens $\mathcal{U}_{i}$ for $i \ge 0$.

$\underline{Claim:}$ The collection $\{\mathcal{U}_i: i \ge 0\}$ is a large scale basis for a metric large scale structure on $X$ with asymptotic dimension at most $n$. Furthermore, the uniformly bounded covers of this metric large scale structure generated by $\mathcal{U}$ are uniformly  bounded in $(X,\mathcal{LSS}_X)$
\begin{proof}[proof of claim:]
Indeed, let $k > l \ge 0$ be integers and notice that $\mathcal{U}_k$ and $\mathcal{U}_l$ has the property that $\mathcal{U}_k \cup \mathcal{U}_l \cup st(\mathcal{U}_l,\mathcal{U}_k) $ refines $\mathcal{U}_{k+1}$. To see this, observe that $\mathcal{U}_l$ is coarsened by $\mathcal{U}_{k-1}$ and we have $\mathcal{U}_{k+1}$ coarsens $st(\mathcal{U}_{k},\mathcal{U}_{k-1})$ which in turn coarsens $st(\mathcal{U}_{k},\mathcal{U}_{l})$. In this case, $st(\mathcal{U}_{k},\mathcal{U}_{l})$ coarsens $st(\mathcal{U}_{k},\mathcal{U}_{l}) \cup \mathcal{U}_k \cup \mathcal{U}_l$ because the collection $\{\mathcal{U}_i: i \ge 0\}$ is nested. 

$\{\mathcal{U}_i: i \ge 0\}$ is countable and so generates a metric large scale structure by proposition \ref{metrizability}. Use the fact that every element of the large scale basis generating this metric structure  has multiplicity at most $n+1$  to see that it has dimension at most $n$.
\end{proof}

$X$ with the metric structure generated by $\mathcal{U}$ is metric of dimension at most $n$. There exists a uniformly bounded cover $\mathcal{V} = \bigcup_{i = 1}^{n+1} \mathcal{V}_i$ of $X$ with the metric structure where $st(\mathcal{V}_i,\mathcal{U})$ is a disjoint collection for each $i = 1,2,\hdots , n$. $\mathcal{V}$ is uniformly bounded in $(X,\mathcal{LSS}_X)$ which completes the proof.

\end{proof}

\begin{Remark}
\v Ziga Virk and I are finishing up a work \cite{AZ} in which we formalize the method in the previous proof.
\end{Remark}

\begin{Theorem}\label{asdim}
Let $f:X\to Y$  be coarse and coarsely $n$-to-$1$. $X$ has finite asymptotic dimension if and only if $Y$ has finite asymptotic dimension. 
\end{Theorem}
\begin{proof}
$(\Rightarrow)$ Say, $asdim_d(X) \leq m-1$. Let $\mathcal{U}$ be a uniformly bounded cover of $Y$. Let $\mathcal{W}$ be a uniformly bounded cover of $X$ such that the preimage of element of $\mathcal{U}$ can be covered by at most $n$ elements of $\mathcal{W}$. Let $\mathcal{V} = \bigcup_{i=0}^{m}\mathcal{V}_i$ be a cover of $X$ such that $\mathcal{V}_i$ is $\mathcal{W}$-disjoint for each $i = 0,1,2,\hdots,m$. Consider the cover $\mathcal{W}_0 := st(f(\mathcal{V}),\mathcal{U})$ of $Y$. We claim that this cover has multiplicity bounded by $n\cdot m$. Consider the multiplicity of $st(f(\mathcal{V}_i),\mathcal{U})$. Let $y \in Y$ and consider how many elements of $st(f(\mathcal{V}_i),\mathcal{U})$ could contain $y$. $f^{-1}(y)$ can be covered by at most $n$ elements of $\mathcal{W}$ and $\mathcal{V}_i$ is $\mathcal{W}-$disjoint which means that $y$ could only belong to at most $n$ elements of $st(f(\mathcal{V}_i),\mathcal{U})$. Notice that this implies $y$ belongs to at most $n\cdot m$ elements of $\mathcal{W}_0$.

$(\Leftarrow)$ Say $asdim_d(Y) \leq m-1$. Let $\mathcal{U}$ be a uniformly bounded cover of $X$. Then $f(\mathcal{U})$ is a uniformly bounded cover of $Y$ and so we can find a cover $\mathcal{V} = \bigcup_{i=1}^{m}\mathcal{V}_i$ where $\mathcal{V}_i$ is $f(\mathcal{U})$-disjoint. Let $\mathcal{W}_0$ be a uniformly bounded cover of $X$ such that for every $B\in \mathcal{V}$ there exists $B_1,B_2,\hdots, B_n \in \mathcal{W}_0$ with $f^{-1}(B) = \bigcup_{i=1}^n B_i$. Let $\mathcal{W}_i = \{B_1,B_2,\hdots,B_n: B\in \mathcal{V}_i\}$ and let $\mathcal{W} = \bigcup_{i=1}^{n}\mathcal{W}_i$ 

Consider $\mathcal{B} = st(\mathcal{W},\mathcal{U}) = \bigcup_{i=1}^n st(\mathcal{W}_i,\mathcal{U})$; it is a coarsening of $\mathcal{U}$. We need only show it has bounded multiplicity. Fix $i\in \{1,2,\hdots n\}$ and consider the multiplicity of $st(\mathcal{W}_i,\mathcal{U})$. Notice that for distinct $A,B \in \mathcal{V}_i$ their preimages $A_i$ and $B_j$ are $\mathcal{U}$ disjoint for $i,j =1,2,\hdots n.$ It follows that the multiplicity of $st(\mathcal{W}_i,\mathcal{U})$ is at most $n$. Thus the multiplicity of $\mathcal{B}$ is at most $mn$ and so $asdim(X) \leq m\cdot n$.
\end{proof}

The proof of \ref{asdim} could be used verbatim to prove the following

\begin{Theorem}
Let $f:X\to Y$  be coarse and coarsely $n$-to-$1$. $X$ is large scale finitistic if and only if $Y$ is large scale finitistic. 
\end{Theorem}

\section{Invariance of Large Scale Weak Paracompactness}

In \cite{CDV1} the authors show that one can define a metric space $X$ to be large scale weakly paracompact if for every uniformly bounded cover $\mathcal{U}$ of $X$ there exists a uniformly bounded cover $\mathcal{V}$ so that each element of $\mathcal{U}$ intersects at most finitely many elements of $\mathcal{V}$. The following definition of large scale weak paracompactness is a natural extension of this idea.

\begin{Definition}
An ls space $X$ is \textbf{large scale weakly paracompact} for every uniformly bounded cover $\mathcal{U}$ of $X$ there exists a uniformly bounded cover $\mathcal{V}$ so that each element of $\mathcal{U}$ intersects at most finitely many elements of $\mathcal{V}$.
\end{Definition}

\begin{Theorem}
Let $X$ and $Y$ be spaces and $f:X \to Y$ coarse and coarsely finite-to-$1$. $X$ is large scale weakly paracompact then so is $Y$.
\end{Theorem}
\begin{proof}
Let $\mathcal{U}$ be a uniformly bounded cover of $Y$. Let $\mathcal{V}$ be a uniformly bounded cover of $X$ such that for every $U\in \mathcal{U}$ there exists $U_1,U_2,\hdots,U_n \in \mathcal{V}$ with $f^{-1}(U)= \bigcup_{i=1}^{n}U_i$. Let $\mathcal{W}$ be a uniformly bounded covering of $X$ for which every element of $\mathcal{U}$ intersects only finitely many elements of $\mathcal{W}$. Consider the uniformly bounded cover $f(\mathcal{W})$ of $Y$. We claim that each element of $\mathcal{V}$ intersects only finitely many elements of $f(\mathcal{W})$. To see this, let $V\in \mathcal{V}$. We have that $f^{-1}(V)\subset \bigcup_{i=1}^{n}V_i$ for $V_1,V_2, \hdots,V_n \in \mathcal{U}$. $V_i$ intersects at most finitely many, say at most $m$ for some $m \ge 1$, elements of $\mathcal{W}$ for each $i = 1,2,\hdots, n$. It follows that $\mathcal{U}$ intersects at most $n\cdot m$ elements of $f(\mathcal{U})$.
\end{proof}


\begin{thebibliography}{99}

\bibitem{AD}
Kyle Austin and Jerzy Dydak
\emph{Partitions of Unity and Coverings},
Topology and Its Applications, Volume 173, 15 August 2014, pages 74-82

\bibitem{AZ}
Kyle Austin and \v Ziga Virk
\emph{Coarse Metric Approximation and Dimension Raising},
preprint


\bibitem{BD}
G. Bell and A. Dranishnikov
\emph{Aysmptotic Dimension in Bedlewo},
Top. Proc. 38 (2011), 209?36. 

\bibitem{CDV1}
Matija Cencelj, Jerzy Dydak and Ale\v s Vavpeti\v c
\emph{Coarse Amenability Versus Paracompactness},
Journal of Topology and Analysis, Vol.06(2014)No.01,pp.125-152

\bibitem{CDV2}
Matija Cencelj, Jerzy Dydak and Ale\v s Vavpeti\v c
\emph{Large Scale Versus Small Scale},
Recent Progress in General Topology III, Hart K.P.; van Mill, Jan; Simon, P(Eds)(Atlantic Press, 2014)pp.165--204

\bibitem{DyCAD}
Jerzy Dydak 
\emph{Coarse Amenability and Discreteness},
arXiv preprint arXiv:1307.3943 

\bibitem{DH}
J. Dydak, C.S. Hoffland
\emph{An Alternative Definition of Coarse Structures},
Topology and its Applications 155 (9), 1013-1021

\bibitem{DV}
J. Dydak and Z. Virk
\emph{Preserving Coarse Properties},
arXiv preprint arXiv:1506.08287

\bibitem{MV}
 T. Miyata, \v Z. Virk {\em  Dimension Raising Maps in a Large Scale},  
Fundamenta Mathematicae Vol. 223 (2013), 83-98.

\bibitem{R}
 John Roe {\em  Lectures on Coarse Geometry}, 
University Lecture Series, 31. American Mathematical Society, Providence, RI, 2003.


\end{thebibliography}
\end{document}